\documentclass{amsart}
\usepackage{amssymb,amscd}
\DeclareMathSymbol{\twoheadrightarrow} {\mathrel}{AMSa}{"10}

\def\Q{{\mathbb Q}}
\def\Z{{\mathbb Z}}

\def\F{{\mathbb F}}

                    \def\G{{\mathfrak G}}

            \def\mult{\mathrm{mult}}
\def\Gal{\mathrm{Gal}}
                                            \def\Pic{\mathrm{Pic}}

\def\tr{\mathrm{tr}}

\def\End{\mathrm{End}}
\def\Aut{\mathrm{Aut}}
\def\Hom{\mathrm{Hom}}

\def\cl{\mathrm{Cliff}}

\def\fchar{\mathrm{char}}

                                            \def\H{\mathrm{H}}
\def\dim{\mathrm{dim}}

                                                             \def\Br{\mathrm{Br}}
                                                                            \def\NS{\mathrm{NS}}
                                                                              \def\PH{\mathrm{PH}}

                                                                                 \def\non{\mathrm{non}}

      \def\sep{\mathrm{sep}}

\def\et{{\mathrm{\acute et}}}
\def\G{{\mathbb G}}
\def\lra{{\longrightarrow}}

\newtheorem{thm}{Theorem}[section]

\newtheorem{cor}[thm]{Corollary}
\newtheorem{prop}[thm]{Proposition}

\theoremstyle{definition}

           \newtheorem{rem}[thm]{Remark}

\title[A finiteness theorem for the Brauer group]
{A finiteness theorem for the Brauer group of K3 surfaces in odd characteristic}
\author{Alexei N. Skorobogatov}
\address{Department of Mathematics, South Kensington Campus, Imperial College, London, SW7 2BZ England, United Kingdom}
\address{Institute for the Information Transmission Problems, Russian Academy of Sciences, 19 Bolshoi Karetnyi, Moscow, 127994 Russia}
\email{a.skorobogatov@imperial.ac.uk}
\author[Yuri G. Zarhin]{Yuri G. Zarhin}
\thanks{This work was partially supported by a grant from the Simons Foundation (\#246625 to Yuri Zarkhin).}

\address{Department of Mathematics, Pennsylvania State University,
University Park, PA 16802, USA}
\address{Department of Mathematics, The Weizmann Institute of Science,
 POB 26,  Rehovot 7610001, Israel}

\address{Institute of Mathematical Problems of Biology, Russian Academy of
Sciences, Pushchino, Moscow Region,
Russia}

 \email{zarhin\char`\@math.psu.edu}
\begin{document}
\begin{abstract}
Let $k$ be a field finitely generated over the finite field $\F_p$ 
of odd characteristic $p$.
For any K3 surface $X$ over $k$ we prove that the 
prime to $p$ component of the cokernel of the natural map 
$\Br(k)\to\Br(X)$ is finite. 
\end{abstract}

\maketitle
\section{Introduction}
Let $k$ be a field with an algebraic closure $\bar{k}$, and let $\bar{k}_{\sep}
\subset \bar{k}$ be the separable closure of $k$ in $\bar k$. Let 
$\Gal(k)=\Gal(\bar{k}_{\sep}/k)$ be the absolute Galois
group of $k$, and let $\Br(k)$ be the Brauer group of $k$.

For an irreducible, smooth, projective algebraic variety $X$ over $k$ 
we write $\bar{X}$ for the variety
$X\times_k \bar{k}_{\sep}$ over $\bar{k}_{\sep}$. 
We write $\Br(X)=\H^2_\et(X,\G_m)$ for the Brauer--Grothendieck group of $X$, and $\Br_0(X)$ for the image of the canonical homomorphism $\Br(k)\to \Br(X)$ induced by the structure map.

For an abelian group $B$ and a positive integer $n$ we write
$B/n=B\otimes\Z/n$.

Let $\ell$ be a prime. For an abelian periodic group $B$ let
$B(\non-\ell)$ be the subgroup of $B$ consisting of the 
elements of order prime to $\ell$.  
We write $T_{\ell}(B)$ for the $\ell$-adic {\sl Tate module} of $B$; 
it has the natural structure of a $\Z_{\ell}$-module and is torsion-free. If $B$ has only finitely many elements of order $\ell$, then $T_{\ell}(B)$ is a free $\Z_{\ell}$-module of finite rank.
See \cite[p. 485]{SZ} for more details.

Let $X$ be a K3 surface over $k$. It is known that the N\'eron--Severi group
$\NS(\bar{X})=\Pic(\bar{X})$ is a free abelian group of rank $22$.
If $\ell\ne \fchar(k)$, then $\H^2_\et(\bar{X},\Z_{\ell}(1))$ is a free $\Z_{\ell}$-module of rank $22$, and the natural map
$$\H_{\et}^2(\bar{X},\Z_{\ell}(1))/\ell\ \tilde\lra \ \H_{\et}^2(\bar{X},\mu_{\ell})$$
is an isomorphism of $\Gal(k)$-modules. 
By Grothendieck \cite[III.8.2]{Gro} and Tate \cite{TateDix} there is
a short exact sequence of 
$\Gal(k)$-modules, which are also {\sl free} $\Z_{\ell}$-modules of finite rank:
$$0\lra\NS(\bar{X})\otimes \Z_{\ell} \lra \H_{\et}^2(\bar{X},\Z_{\ell}(1)) \lra
 T_{\ell}(\Br(\bar{X})) \lra 0. \eqno{(1)}$$
We identify $\NS(\bar{X})\otimes \Z_{\ell}$ with its image in $\H_{\et}^2(\bar{X},\Z_{\ell}(1))$.
Then (1) shows that $\NS(\bar{X})\otimes \Z_{\ell}$ is a {\sl saturated} $\Z_{\ell}$-submodule of $\H_{\et}^2(\bar{X},\Z_{\ell}(1))$, that is,
the quotient is torsion-free.
It is well known that there exists a finite Galois field extension $k'/k$ such that the open finite index subgroup $\Gal(k')\subset \Gal(k)$ acts trivially on $\NS(\bar{X})$ and hence on $\NS(\bar{X})\otimes \Z_{\ell}$.

We {\sl define} $\NS(X)$ as the Galois invariant subgroup
$\NS(X):=\NS(\bar X)^{\Gal(k)}$. (Note that the image of $\Pic(X)$
in $\NS(X)$ is a subgroup of finite index.)
In particular, $\NS(X)$ is a saturated $\Z$-submodule of $\NS(\bar{X})$.
We have
$$\NS(X)\otimes \Z_{\ell} \subset \H_{\et}^2(\bar{X},\Z_{\ell}(1))^{\Gal(k)}\subset \H_{\et}^2(\bar{X},\Z_{\ell}(1)).$$
Since $\NS(X)\otimes \Z_{\ell}$ is a saturated $\Z_{\ell}$-submodule of 
$\NS(\bar X)\otimes \Z_{\ell}$, it is also a saturated $\Z_{\ell}$-submodule of 
$\H_{\et}^2(\bar{X},\Z_{\ell}(1))$. The intersection pairing
$$\NS(\bar{X}) \times \NS(\bar{X}) \to \Z, \quad\quad \alpha,\beta \mapsto \alpha\cdot \beta$$
is non-degenerate, see \cite[Prop. 3, p. 64]{Matsusaka}
and also \cite[Lemma V.3.27]{Milne}.
In particular, the discriminant $d_X$ of this pairing is a non-zero integer. 
On the other hand, there is the Poincar\'e duality pairing
$$e_{X,\ell}: \H_{\et}^2(\bar{X},\Z_{\ell}(1))\times \H_{\et}^2(\bar{X},\Z_{\ell}(1)) \to \Z_{\ell},$$
which is perfect and Galois-invariant. The restriction of $e_{X,\ell}$ to $\NS(\bar{X})$ coincides with the intersection pairing. 
We define the group of transcendental cycles $T(\bar X)_{\ell}$ as the orthogonal complement to $\NS(\bar{X})\otimes \Z_{\ell}$ in $\H_{\et}^2(\bar{X},\Z_{\ell}(1))$ with respect to  $e_{X,\ell}$. It is clear that $T(\bar X)_{\ell}$ is a saturated Galois-stable 
$\Z_{\ell}$-submodule of $\H_{\et}^2(\bar{X},\Z_{\ell}(1))$. 
It is also clear that if $\ell$ does not divide $d_X$, then
$$\H_{\et}^2(\bar{X},\Z_{\ell}(1))=(\NS(\bar{X})\otimes \Z_{\ell})\oplus T(\bar X)_{\ell}.$$
For any $\ell$ we have
$$(\NS(\bar{X})\otimes \Z_{\ell})\ \cap \ T(\bar X)_{\ell}=0,$$
and the direct sum $(\NS(\bar{X})\otimes \Z_{\ell})\oplus T(\bar X)_{\ell}$ is a subgroup of finite index in $\H_{\et}^2(\bar{X},\Z_{\ell}(1))$. 
It is clear that $T(\bar X)_{\ell}/\ell$ is a $\Gal(k)$-submodule of
$\H_{\et}^2(\bar{X},\mu_{\ell})$.

Let us consider the $\Q_{\ell}$-vector space
$\H_{\et}^2(\bar{X},\Q_{\ell}(1)):=\H_{\et}^2(\bar{X},\Z_{\ell}(1))\otimes_{\Z_{\ell}}\Q_{\ell}$
with its natural structure of a $\Gal(k)$-module. We have a natural embedding of Galois modules
$\NS(\bar{X})\otimes\Q_{\ell} \subset \H_{\et}^2(\bar{X},\Q_{\ell}(1))$,
from which we obtain
$$\NS(X)\otimes\Q_{\ell}\subset \H_{\et}^2(\bar{X},\Q_{\ell}(1))^{\Gal(k)}.$$

K. Madapusi Pera \cite[Thm. 1]{Pera}
proved that if $k$ is finitely generated over $\F_p$ and $p>2$, then
$$\NS(X)\otimes \Q_{\ell} = \H_{\et}^2(\bar{X},\Q_{\ell}(1))^{\Gal(k)}.$$
This is an important special case of the celebrated Tate conjecture 
on algebraic cycles \cite{Tate0}.
Closely related results were previously obtained by
D. Maulik \cite{Maulik} and F. Charles \cite[Cor. 2]{Charles}.
Since $\NS(X)\otimes \Z_{\ell}$ is a saturated $\Z_{\ell}$-submodule  in $\H_{\et}^2(\bar{X},\Z_{\ell}(1))$, the theorem of 
Madapusi Pera is equivalent to 
$$\NS(X)\otimes \Z_{\ell} = \H_{\et}^2(\bar{X},\Z_{\ell}(1))^{\Gal(k)}.$$
Another  restatement of the same result is
$$T(\bar X)_{\ell}^{\Gal(k)}=0.$$
Note that this still holds if $k$ is replaced by any
finite separable field extension. 

Our results are based on the Tate conjecture for K3 surfaces and 
the existence of the Kuga--Satake
variety in odd characteristic, as established by Madapusi Pera in
\cite[Thm. 4.17]{Pera}, as well as on the results of one of the authors
\cite{ZarhinMZ1}.

\begin{thm}
\label{semisimpleK3}
Let $k$ be a field finitely generated over $\F_p$, where $p\not=2$,
and let $X$ be a K3 surface over $k$.
For any prime $\ell\not=p$ the Galois module
$\H_{\et}^2(\bar{X},\Q_{\ell}(1))$ is  semisimple.
\end{thm}

This is proved in Theorem \ref{finiteK3}~(i).
Our main results are the two following statements whose proofs
can be found in the final section of this note.

\begin{thm}
\label{tatefiniteK3}
Let $k$ be a field finitely generated over $\F_p$, where $p\not=2$, 
and let $X$ be a K3 surface over $k$. Then
for all but finitely many $\ell$ the Galois module 
$\H_{\et}^2(\bar{X},\mu_{\ell})$ is semisimple and
$\H_{\et}^2(\bar{X},\mu_{\ell})^{\Gal(k)}=\NS(X)/\ell$.
\end{thm}

Using the method of \cite{SZ}
we deduce from the above results the following finiteness statement.

\begin{thm}
\label{BrauerF}
Let $k$ be a field finitely generated over $\F_p$, where $p\not=2$, and 
let $X$ be a K3 surface over $k$. Then
the group $[\Br(X)/\Br_0(X)](\non-p)$ is finite.
\end{thm}

For a K3 surface $X$ over a finitely generated field of 
characteristic zero, the finiteness of $\Br(X)/\Br_0(X)$ 
was proved in \cite{SZ}.

\section{Abelian varieties}
Let $A$ be an abelian variety over a field $k$. We write $\End(A)$
for the ring of endomorphisms of $A$.
For a positive integer $n$ not divisible by $\fchar(k)$, 
we write $A_n$ for the kernel of multiplication by $n$ in $A(\bar{k})$. 
Then $A_n$ is a $\Gal(k)$-submodule of $A(\bar{k}_{\sep})$. 

Let $\ell\not=\fchar(k)$ be a prime. Recall that
the $\ell$-adic Tate module $T_{\ell}(A)$ is the
projective limit of abelian groups $A_{\ell^i}$, where the transition maps
$A_{\ell^{i+1}}\to A_{\ell^i}$ are multiplications by $\ell$. It is
well known \cite{Serre} that $T_{\ell}(A)$ is a free
$\Z_{\ell}$-module of rank $2\dim(A)$ equipped with a natural
continuous action
$$\rho_{\ell,A}:\Gal(k) \to \Aut_{\Z_{\ell}}(T_{\ell}(A)).$$
We write $V_{\ell}(A)=T_{\ell}(A)\otimes_{\Z_{\ell}}\Q_{\ell}$, and
view $T_{\ell}(A)$ as a $\Z_{\ell}$-lattice in
$V_{\ell}(A)$. Then $\rho_{\ell,A}$ gives rise to the $\ell$-adic representation
$$\rho_{\ell,A}: \Gal(k) \to \Aut_{\Z_{\ell}}(T_{\ell}(A))\subset
\Aut_{\Q_{\ell}}(V_{\ell}(A)).$$ 
There are natural canonical isomorphisms of Galois modules
$T_{\ell}(A)/\ell^i \cong A_{\ell^i}$ for all $i\geq 1$. 
If we forget about the Galois action, then we can naturally 
identify the $\Z_{\ell}$-modules $T_{\ell}(A)$
and $T_{\ell}(\bar{A})$, 
and the $\Q_{\ell}$-vector spaces $V_{\ell}(A)$ and $V_{\ell}(\bar{A})$.

There is a natural embedding of $\Z$-algebras
$\End(A) \hookrightarrow \End(\bar{A})$,
whose image is a saturated subgroup \cite[Sect. 4, p. 501]{SerreTate}. 
We will identify $\End(A)$ with its image in $\End(\bar{A})$.
The natural action of $\Gal(k)$ on $\End(\bar{A})$ is continuous when
$\End(\bar{A})$ is given discrete topology, so 
the image of $\Gal(k)$ in $\Aut(\End(\bar{A}))$ is finite.
The ring $\End(A)$ coincides with the Galois-invariant
subring $\End(\bar{A})^{\Gal(k)}$, see \cite{Silverberg}. 
In other words, there is a finite Galois field extension 
$k'/k$ such that $\Gal(k) \to \Aut(\End(\bar{A}))$
factors through $\Gal(k) \twoheadrightarrow \Gal(k'/k)$
and
$$\End(A)=\End(\bar{A})^{\Gal(k'/k)}.$$

There are natural embeddings of  $\Z_{\ell}$-algebras
$$\End(A)\otimes \Z_{\ell}\subset \End(\bar{A})\otimes \Z_{\ell}
\subset \End_{\Z_{\ell}}(T_{\ell}(A)),$$
of $\Q_{\ell}$-algebras
$$\End({A})\otimes \Q_{\ell} \subset  \End(\bar{A})\otimes \Q_{\ell}\subset
\End_{\Q_{\ell}}(V_{\ell}(A))$$
and of $\F_{\ell}$-algebras
$$\End({A})\otimes \F_\ell \subset  \End(\bar{A})\otimes \F_\ell \subset
\End_{\F_{\ell}}(A_{\ell}).$$
The image of $\End(A)\otimes\Z_{\ell}$ in $\End_{\Z_{\ell}}(T_{\ell}(A))$
is obviously contained in the centraliser 
$\End_{\Gal(k)}(T_{\ell}(A))$ of $\Gal(k)$ in
$\End_{\Z_{\ell}}(T_{\ell}(A))$. Similar statements
hold if $\Z_\ell$ is replaced by $\Q_\ell$ or $\F_\ell$.  

For each prime $\ell \ne \fchar(k)$ let us consider 
the $\Q_{\ell}$-bilinear  {\sl trace form}
$$\psi_{\ell}: \End_{\Q_{\ell}}(V_{\ell}(A)) \times \End_{\Q_{\ell}}(V_{\ell}(A)) \to \Q_{\ell}, \ u,v \mapsto \tr(uv).$$
This form is Galois-invariant, symmetric and
non-degenerate. The same formula defines
a perfect (unimodular) $\Z_{\ell}$-bilinear 
form on $\End_{\Z_{\ell}}(T_{\ell}(A))$. 
The restriction of this form to $\End(\bar{A})$ 
takes values in $\Z$ and does not depend on the choice of $\ell$, see
\cite[Sect. 19, Thm. 4]{Mumford}. 

Let us choose a polarisation on $A$ and write 
$u \mapsto u'$ for
the corresponding Rosati involution on $\End(\bar{A})\otimes \Q$, see 
\cite[Sect. 20]{Mumford}. Then
$\tr(u'u)$ is a {\sl positive} rational number if $u\not=0$,
see \cite[Ch. V, Sect. 3, Thm. 1]{Lang} or \cite[Sect. 21, Thm. 1]{Mumford}.
This implies that the $\Gal(k)$-invariant form
$$\psi_{\ell}: \End(\bar{A})\times \End(\bar{A})\to \Z$$
is non-degenerate (but not  necessarily perfect) and does not depend on $\ell$.
Since the Rosati involution commutes with the action of $\Gal(k)$,
and $\End(A)=\End(\bar{A})^{\Gal(k)}$, the bilinear form
$$\psi_{\ell}: \End(A)\times \End(A)\to \Z$$
is also non-degenerate, and therefore its discriminant $d_A$ is a {\sl non-zero} integer that does not depend on $\ell$. 

The following assertion was proved by one of the authors 
\cite{ZarhinIz,ZarhinMZ1} in $\fchar(k)>2$, by Faltings 
\cite{Faltings1,Faltings2} in $\fchar(k)=0$ and by Mori \cite{MB} 
in $\fchar(k)=2$. (See also \cite{ZarhinP,ZarhinCEJM14}.)

\begin{thm}
\label{tateSS}
Let $k$ be a field finitely generated over its prime subfield and
let $A$ be an abelian variety over $k$. For a prime $\ell\not=\fchar(k)$
we have

\smallskip

{\rm (i)} 
$\End(A)\otimes \Z_{\ell}= \End_{\Gal(k)}(T_{\ell}(A))$ and
$\End(A)\otimes \Q_{\ell}= \End_{\Gal(k)}(V_{\ell}(A))$;

\smallskip

{\rm (ii)} the Galois module $V_{\ell}(A)$ is semisimple.
\end{thm}

\begin{rem}
When $k$ is finite, the assertion (i) of Theorem \ref{tateSS} was proved by Tate \cite{Tate}, who conjectured that it holds for an arbitrary finitely generated field. He also proved that the two formulae in (i) are equivalent for any given $k,A,\ell$. The semisimplicity of $V_{\ell}(A)$ for finite $k$ was established earlier by Weil \cite{Mumford}.
\end{rem}

The following assertion was proved in \cite{ZarhinMZ2,ZarhinInv,ZarhinCEJM14}.

\begin{thm}
\label{finite}
Let $k$ be a field finitely generated over its prime subfield,
and let $A$ be an abelian variety over $k$. 
Then for all but finitely many primes $\ell\ne \fchar(k)$ 
the Galois module $A_{\ell}$ is semisimple and
$\End(A)/\ell= \End_{\Gal(k)}(A_{\ell})$.
\end{thm}

\begin{prop}
\label{tech}
Let $k$ be a field finitely generated over its prime subfield, and let $A$ be an abelian variety over $k$. Let $P$ be an infinite set of primes that does not contain $\fchar(k)$. Suppose that for each $\ell\in P$ we are given 
a $\Gal(k)$-submodule $U_{\ell}\subset  \End_{\Z_{\ell}}(T_{\ell}(A))$ such that $U_{\ell}$ is a saturated $\Z_\ell$-submodule of $\End_{\Z_{\ell}}(T_{\ell}(A))$
and ${U_{\ell}}^{\Gal(k)}=0$.
Then for all but finitely many $\ell\in P$ we have
$(U_{\ell}/\ell)^{\Gal(k)}=0$.
\end{prop}
\begin{proof}
By Theorem \ref{tateSS}(ii) the Galois module $V_{\ell}(A)$ is 
semisimple, and by a theorem of Chevalley
\cite[p. 88] {Chev} this implies that 
$\End_{\Q_{\ell}}(V_{\ell}(A))$ is also semisimple.
The Galois submodules $U_{\ell}\otimes\Q_{\ell}$ and 
$\End_{\Gal(k)}(V_{\ell}(A))$ 
are orthogonal in $\End_{\Q_{\ell}}(V_{\ell}(A))$ 
with respect to the Galois-invariant bilinear form $\psi_{\ell}$,
because otherwise $(U_{\ell}\otimes\Q_{\ell})^{\Gal(k)}\not=0$, 
which contradicts the assumption ${U_{\ell}}^{\Gal(k)}=0$. 
It follows that $\End(A)\otimes\Z_{\ell}$ and 
$U_{\ell}$ are orthogonal, too.

Now assume that $\ell\in P$ does not divide $d_A$. 
Then the $\Z_{\ell}$-bilinear symmetric Galois-invariant form
$$\psi_{\ell}:\ \End(A)\otimes\Z_{\ell} \ \ \times \ \ \End(A)\otimes\Z_{\ell}
\ \to  \ \Z_{\ell}$$
is perfect. Thus the Galois module $\End_{\Z_{\ell}}(T_{\ell}(A))$ 
splits into a direct sum
$$\End_{\Z_{\ell}}(T_{\ell}(A))=(\End(A)\otimes\Z_{\ell}) \oplus S_{\ell},$$
where $S_\ell$ is the orthogonal complement to $\End(A)\otimes\Z_{\ell}$.
It is clear that $S_{\ell}$ is a Galois-stable saturated 
$\Z_{\ell}$-submodule of $\End_{\Z_{\ell}}(T_{\ell}(A))$.
We have natural isomorphisms of Galois modules
$$ \End_{\Z_{\ell}}(T_{\ell}(A))/\ell =
\End_{\Z_{\ell}/\ell}(T_{\ell}(A)/\ell)= \End_{\F_{\ell}}(A_{\ell}).$$
Thus reducing modulo $\ell$, we obtain a Galois-invariant decomposition
$$ \End_{\F_{\ell}}(A_{\ell})=\End(A)/\ell 
\oplus S_{\ell}/\ell.$$ 
This gives rise to
$$\End_{\Gal(k)}(A_{\ell})=\End(A)/\ell\oplus 
(S_{\ell}/\ell)^{\Gal(k)}.$$
Theorem \ref{finite} implies that for all but finitely many 
$\ell\in P$ we have $(S_{\ell}/\ell)^{\Gal(k)}=0$.
Since $U_{\ell}$ is orthogonal to $\End(A)\otimes\Z_{\ell}$, 
we have $U_{\ell} \subset S_{\ell}$. Moreover,
$U_{\ell}$ is a saturated $\Z_{\ell}$-submodule of $S_{\ell}$,
and this implies that the natural map $U_{\ell}/\ell \to S_\ell/\ell$
is injective. We conclude that $(U_\ell/\ell)^{\Gal(k)}=0$
for all but finitely many $\ell\in P$.
\end{proof}

Recall if $\fchar(k)$ does not divide a positive integer $n$, then
we have a canonical isomorphism of Galois modules
$$\H_{\et}^1(\bar{A},\Z/n)=\Hom_{\Z/n}(A_n,\Z/n).$$
For a prime $\ell\not=\fchar(k)$ we deduce canonical isomorphisms
of Galois modules
$$ \H_{\et}^1(\bar{A},\Z_{\ell})=\Hom_{\Z_{\ell}}(T_{\ell}(A),\Z_{\ell}),\quad
\H_{\et}^1(\bar{A},\Q_{\ell})=\Hom_{\Q_{\ell}}(V_{\ell}(A),\Q_{\ell}),$$
$$ \End_{\Z_{\ell}}( \H_{\et}^1(\bar{A},\Z_{\ell}))\cong  \End_{\Z_{\ell}}(T_{\ell}(A)),$$
where the last identification is an anti-isomorphism of rings. 
This allows us to restate Proposition \ref{tech} in the following form.

\begin{cor}
\label{techCOH}
Let $k$ be a field finitely generated over its prime subfield, and let $A$ be an abelian variety over $k$. Let $P$ be an infinite set of primes that does not contain $\fchar(k)$. Suppose that for each $\ell\in P$ we are given a  $\Gal(k)$-submodule $$U_{\ell}\subset \End_{\Z_{\ell}} (\H_{\et}^1(\bar{A},\Z_{\ell}))$$ such that
$U_{\ell}$ is a saturated $\Z_\ell$-submodule of $\End_{\Z_{\ell}} (\H_{\et}^1(\bar{A},\Z_{\ell}))$
and ${U_{\ell}}^{\Gal(k)}=0$.
Then for all but finitely many $\ell\in P$ we have 
$(U_{\ell}/\ell)^{\Gal(k)}=0$.
\end{cor}

\begin{thm}
\label{SScohom}
Let $k$ be a field finitely generated over its prime subfield, and 
let $A$ be an abelian variety over $k$. Then

\smallskip

{\rm (i)} the $\Gal(k)$-module 
$\End_{\Q_{\ell}}( \H_{\et}^1(\bar{A},\Q_{\ell}))$ is semisimple for any prime $\ell \ne \fchar(k)$;

\smallskip

{\rm (ii)} the $\Gal(k)$-module $\End_{\F_{\ell}}(\H_{\et}^1(\bar{A},\Z/\ell))$ is semisimple for all but finitely many primes $\ell\ne \fchar(k)$.
\end{thm}
\begin{proof} (i) We have a $\Q_{\ell}$-algebra anti-isomorphism
$$\End_{\Q_{\ell}}(V_{\ell}(A))\cong \End_{\Q_{\ell}}( \H_{\et}^1(\bar{A},\Q_{\ell})),$$
which is also an isomorphism of Galois modules. By Theorem \ref{tateSS}
the Galois module $V_{\ell}(A)$ is semisimple, and then Chevalley's theorem
\cite[p. 88] {Chev}  implies the semisimplicity of $\End_{\Q_{\ell}}(V_{\ell}(A))$.

(ii) We have an $\F_{\ell}$-algebra anti-isomorphism
$$ \End_{\F_{\ell}}( \H_{\et}^1(\bar{A},\Z/\ell))\cong  \End_{\F_{\ell}}(A_{\ell}),$$
which is also an isomorphism of Galois modules.
By Theorem \ref{finite} the Galois module $A_\ell$ is
semisimple for all but finitely many $\ell$.
By a theorem of Serre \cite[Cor. 1]{SerreSS}, if $\ell$ is greater than
$2\,\dim_{\F_{\ell}}(A_{\ell})-2=4\,\dim(A)-2$,
then the semisimplicity of $A_{\ell}$ implies the  semisimplicity of
the Galois module $\End_{\F_{\ell}}(A_{\ell})$. 
This finishes the proof of (ii).
\end{proof}

\section{Clifford Algebras and Kuga--Satake construction}
\label{clifford}

Let $\Lambda$ be a principal ideal domain and let $E$ 
be a free $\Lambda$-module of finite rank
with a non-degenerate quadratic form 
$q: E \to \Lambda$. 
Let $C(E,q)$ be the Clifford  algebra of $(E,q)$. Let
$\varrho: E \to C(E,q)$
be the canonical $\Lambda$-linear homomorphism satisfying
$$\varrho(x)^2=q(x)\in \Lambda \subset C(E,q).$$
The map $\varrho$ is injective by \cite[Sect. 19.3, Cor. 2 to Thm. 1]{Bourbaki},
so that $\varrho(E)\cong E$.
It follows from \cite[Sect. 19.3,  Thm. 1]{Bourbaki} that $C(E,q)$ is a free $\Lambda$-module of finite rank, and the $\Lambda$-submodule 
$\varrho(E)$ is a direct summand of the
$\Lambda$-module $C(E,q)$; in particular, it is a {\sl saturated} submodule.
Let
$$\mult_L: C(E,q)\to \End_{\Lambda}(C(E,q))$$
be the homomorphism of $\Lambda$-algebras that sends $a \in C(E,q)$ 
to the endomorphism $x \mapsto ax$.
Since $1 \in C(E,q)$, the homomorphism $\mult_L$ is injective. 
In particular, the $\Lambda$-algebras
$C(E,q)$ and $\mult_L(C(E,q))$ are isomorphic.  

We claim that $\mult_L(\varrho(E))$ is a saturated 
$\Lambda$-submodule of $\End_{\Lambda}(C(E,q))$. 
Taking into account that $\varrho(E)$ is saturated in $C(E,q)$,
it is enough to show that $\mult_L(C(E,q))$ is saturated in $\End_{\Lambda}(C(E,q))$.
Indeed, if $a\in C(E,q)$ is such that $\mult_L(a)$ is a divisible 
element in $\End_{\Lambda}(C(E,q))$, then
$\mult_L(a)(x)$ is divisible in $C(E,q)$ for all $x \in C(E,q)$. 
In particular, if we put $x=1$, then we see that
$a=\mult_L(a)(1)$ is divisible in $C(E,q)$. This proves our claim.

\medskip

Now let $\Lambda=\Z_\ell$ and let $E_\ell$ be
a free $\Z_\ell$-module of finite rank
with a non-degenerate quadratic form $q_\ell: E_\ell \to \Z_\ell$.
Let $\Aut(C(E_{\ell},q_{\ell}))$ be the group of automorphisms of 
the $\Z_{\ell}$-algebra $C(E_{\ell},q_{\ell})$.
If there is a continuous $\ell$-adic representation
$\rho_{\ell}: \Gal(k) \to \Aut_{\Z_{\ell}}(E_{\ell})$
that preserves $q_\ell$, then,
by a universal property of Clifford algebras, there is 
a representation
$$\rho_{\ell,\cl}: \Gal(k) \to \Aut(C(E_{\ell},q_{\ell}))$$
such that for any $\sigma \in\Gal(k)$ and any $x\in E_\ell$ we have
$\rho_{\ell,\cl}(\sigma)(\varrho(x))=\varrho(\rho_{\ell}(\sigma)(x))$.
Using \cite[Sect. 19.3,  Thm. 1]{Bourbaki} one easily checks that 
$\rho_{\ell,\cl}$ is continuous.

\begin{prop} \label{semisimplek}
Let $k$ be a field finitely generated over its prime subfield.
Let $P$ be an infinite set of primes not containing $\fchar(k)$.
Assume that for each $\ell\in P$ we are given 
a free $\Z_{\ell}$-module $E_{\ell}$ with a non-degenerate quadratic form
$q_{\ell}: E_{\ell}\to \Z_{\ell}$ and a continuous $\ell$-adic representation
$\rho_{\ell}: \Gal(k) \to\Aut_{\Z_{\ell}}(E_{\ell})$ preserving $q_\ell$.

Suppose that there is an abelian variety $A$ over 
a finite separable field extension $k'$ of $k$, 
and for each prime $\ell\in P$ there is an
isomorphism of $\Z_\ell$-modules
$\iota_{\ell}:C(E_{\ell},q_{\ell})\tilde\lra \H_{\et}^1(\bar A,\Z_{\ell}) $
such that the composition of injective maps of $\Z_\ell$-modules
$$ 
E_{\ell}\stackrel{\varrho}{\xrightarrow{\hspace*{5mm}}} C(E_{\ell},q_{\ell})
\stackrel{\iota_\ell}{\xrightarrow{\hspace*{6mm}}}
\H_{\et}^1(\bar A,\Z_{\ell})
\stackrel{\mult_L}{\xrightarrow{\hspace*{8mm}}}
\End_{\Z_{\ell}}(\H_{\et}^1(\bar A,\Z_{\ell}) )$$
is a homomorphism of $\Gal(k')$-modules.

Then the $\Gal(k)$-module $E_{\ell}\otimes_{\Z_{\ell}}\Q_{\ell}$ 
is semisimple for each $\ell\in P$,
and the $\Gal(k)$-module $E_{\ell}/\ell$ is semisimple 
for all but finitely many $\ell \in P$.
\end{prop}
\begin{proof} Replacing $k'$ by its normal closure over $k$, 
we can assume that $k'$ is a finite Galois extension of $k$,
so that $\Gal(k')\subset\Gal(k)$ is an (open) normal subgroup of index $[k':k]$.
In the beginning of this section we have seen that
$\mult_L (\varrho(E_{\ell}))$ is a saturated $\Z_{\ell}$-submodule of
$\End_{\Z_{\ell}} ( C(E_{\ell},q_{\ell}))$. 
Since $\iota_{\ell}$ is an isomorphism, 
$\mult_L(\iota_{\ell} (\varrho(E_{\ell})))$  is a saturated 
$\Z_{\ell}$-submodule of $\End_{\Z_{\ell}}(\H_{\et}^1(A,\Z_{\ell}))$
isomorphic to $E_{\ell}$ as a  $\Gal(k')$-module.
This immediately implies that the $\Gal(k')$-module
$E_{\ell}/\ell$  is isomorphic to a submodule of
$$\End_{\Z_{\ell}}(\H_{\et}^1(A,\Z_{\ell}) )/\ell=\End_{\F_{\ell}}(\H_{\et}^1(A,\F_{\ell})).$$
The $\Q_{\ell}$-vector space
$E_{\ell}\otimes_{\Z_{\ell}}\Q_{\ell}$ is
a $\Gal(k)$-module. Considered as a $\Gal(k')$-module, 
it is isomorphic to a submodule of
$$\End_{\Z_{\ell}}(\H_{\et}^1(A,\Z_{\ell}) )\otimes_{\Z_{\ell}}\Q_{\ell}=\End_{\Q_{\ell}}(\H_{\et}^1(A,\Q_{\ell})).$$
By applying Theorem \ref{SScohom} to
the abelian variety $A$ over $k'$ we obtain that the 
$\Gal(k')$-module $E_{\ell}\otimes_{\Z_{\ell}}\Q_{\ell}$ 
is semisimple for each $\ell\in P$,
and that the $\Gal(k')$-module $E_{\ell}/\ell$ 
is semisimple for all but finitely many $\ell \in P$.
The semisimplicity of  
$E_{\ell}\otimes_{\Z_{\ell}}\Q_{\ell}$ as a $\Gal(k)$-module follows 
from its semisimplicity as a $\Gal(k')$-module by
\cite[Lemma 5 (b), p. 523]{SerreSS}. By the same lemma,
if we further exclude the finitely many primes dividing $[k':k]$, 
the semisimplicity of the $\Gal(k)$-module $E_{\ell}/\ell$
follows from the semisimplicity of the $\Gal(k')$-module $E_{\ell}/\ell$.
\end{proof}

We finish this section with the following complement to
Proposition \ref{semisimplek}.

\begin{prop} \label{add}
In the assumptions of Proposition \ref{semisimplek}
suppose that for each prime $\ell\in P$ 
we have a $\Gal(k')$-submodule 
$U_{\ell}$ of $E_{\ell}$ such that $U_\ell$ is saturated 
in $E_\ell$ and ${U_{\ell}}^{\Gal(k')}=0$.
Then $(U_{\ell}/\ell)^{\Gal(k')}=0$ 
for all but finitely many $\ell\in P$.
\end{prop}
\begin{proof} Since $U_\ell$ is saturated in $E_\ell$ and 
$\mult_L(\iota_{\ell}( \varrho(E_{\ell})))\cong E_\ell$ is saturated in 
$\End_{\Z_{\ell}}(\H_{\et}^1(A,\Z_{\ell}) )$, we see that
$\mult_L(\iota_{\ell}( \varrho(U_{\ell})))$
is a saturated $\Z_\ell$-submodule of 
$\End_{\Z_{\ell}}(\H_{\et}^1(A,\Z_{\ell}) )$. 
It is $\Gal(k')$-stable and isomorphic to $U_\ell$ as a Galois module.
The proposition now follows from Corollary \ref{techCOH}.
\end{proof}

\section{K3 surfaces}

\begin{thm}
\label{finiteK3}
Let $k$ be a field of odd characteristic $p$ that is finitely generated
over $\F_p$. Let $X$ be a K3 surface over $k$. Then

{\rm (i)} for any prime $\ell\not=p$ the $\Gal(k)$-modules
$\H_{\et}^2(\bar{X},\Q_{\ell}(1))$ and
$T(\bar X)_{\ell}\otimes_{\Z_{\ell}}\Q_{\ell}$
are semisimple;

{\rm (ii)}
for all but finitely many primes $\ell$ the $\Gal(k)$-modules
$\H_{\et}^2(\bar{X},\mu_{\ell})$ and
$T(\bar X)_{\ell}/\ell$ are semisimple;

{\rm (iii)}
for all but finitely many primes $\ell$ we have
$(T(\bar X)_{\ell}/\ell)^{\Gal(k)}=0$.
\end{thm}
\begin{proof}
Choose a polarisation 
$$\xi \in \Pic(X)\subset\NS(\bar X)^{\Gal(k)}\subset 
\H_{\et}^2(\bar{X},\Z_{\ell}(1))^{\Gal(k)}.$$
By the adjunction formula the degree of $\xi$ is even,
so that $(\xi^2)=e_{X,\ell}(\xi,\xi)=2d$,
where $d$ is a positive integer.

Let $\PH_{\et}^2(\bar{X},\Z_{\ell}(1))$ be the orthogonal complement to 
$\xi$ in $\H_{\et}^2(\bar{X},\Z_{\ell}(1))$ with respect to $e_{X,\ell}$.
Clearly, $\PH_{\et}^2(\bar{X},\Z_{\ell}(1))\simeq\Z_\ell^{21}$ 
is a Galois-stable saturated $\Z_{\ell}$-submodule of 
$\H_{\et}^2(\bar{X},\Z_{\ell}(1))$. We have
$$\Z_{\ell}\,\xi\, \cap\, \PH_{\et}^2(\bar{X},\Z_{\ell}(1))=0$$
and the direct sum $\Z_{\ell}\,\xi\, \oplus\, \PH_{\et}^2(\bar{X},\Z_{\ell}(1))$ is a subgroup of index $2d$ in $\H_{\et}^2(\bar{X},\Z_{\ell}(1))$. 
Furthermore, if $\ell$ does {\sl not} divide $2d$, then
$$
\H_{\et}^2(\bar{X},\Z_{\ell}(1))=\Z_{\ell}\,\xi\, \oplus\, \PH_{\et}^2(\bar{X},\Z_{\ell}(1))
\eqno{(2)}
$$
and the restriction
$$e_{X,\ell}:  \PH_{\et}^2(\bar{X},\Z_{\ell}(1))\times  
\PH_{\et}^2(\bar{X},\Z_{\ell}(1)) \to \Z_{\ell}$$
is a perfect Galois-invariant pairing. 

By a theorem of K. Madapusi Pera \cite[Thm. 4.17]{Pera} 
the family of $\ell$-adic representations
$E_\ell:=\PH_{\et}^2(\bar{X},\Z_{\ell}(1))$ preserving
the quadratic form $q_\ell(x)=e_{X,\ell}(x,x)$, where $\ell$
is a prime not equal to $p$, satisfies
the assumption of Proposition \ref{semisimplek}.
It follows that the $\Gal(k)$-module
$$\PH_{\et}^2(\bar{X},\Q_{\ell}(1))=\PH_{\et}^2(\bar{X},\Z_{\ell}(1))\otimes_{\Z_{\ell}}\Q_{\ell}$$
is semisimple for each $\ell\not=p$. This implies that 
the $\Gal(k)$-module
$$\H_{\et}^2(\bar{X},\Q_{\ell}(1)) =\Q_{\ell}\,\xi\, \oplus\, \PH_{\et}^2(\bar{X},\Q_{\ell}(1))$$
is semisimple, and hence so is its $\Gal(k)$-submodule
$T(\bar X)_{\ell}\otimes_{\Z_{\ell}}\Q_{\ell}$,
so (i) is proved. 

Suppose that $(\ell,2d)=1$.
Let $\bar\xi$ be the image of $\xi$ in 
$$\H_{\et}^2(\bar{X},\mu_{\ell}(1))=\H_{\et}^2(\bar{X},\Z_{\ell}(1))/\ell.$$
From (2) we obtain a direct sum decomposition of $\Gal(k)$-modules
$$\H_{\et}^2(\bar{X},\mu_{\ell}(1))=\F_{\ell}\bar\xi \oplus 
\PH_{\et}^2(\bar{X},\Z_{\ell}(1))/\ell,$$
where $\bar\xi$ is $\Gal(k)$-invariant.
By Proposition \ref{semisimplek} the $\Gal(k)$-module
$\PH_{\et}^2(\bar{X},\Z_{\ell}(1))/\ell$
is semisimple for all but finitely many primes $\ell \ne p$.
Hence $\H_{\et}^2(\bar{X},\mu_{\ell}(1))$ is semisimple too.
Since $T(\bar X)_{\ell}$ is a Galois-stable 
saturated $\Z_{\ell}$-submodule 
of $\PH_{\et}^2(\bar{X},\Z_{\ell}(1))$, the $\Gal(k)$-module
$T(\bar X)_{\ell}/\ell$ is isomorphic to a submodule of 
$\PH_{\et}^2(\bar{X},\Z_{\ell}(1))/\ell$. This proves (ii).

Part (iii) follows from Proposition \ref{add} applied to
the family of $\Gal(k)$-submodules 
$$U_\ell:=T(\bar X)_\ell\ \subset \ E_\ell=\PH_{\et}^2(\bar{X},\Z_{\ell}),$$
for all $\ell\not=p$.
Indeed, by the Tate conjecture for K3 surfaces over finitely
generated fields of odd characteristic, proved by Madapusi Pera
\cite[Thm. 1]{Pera} (see also \cite{Maulik} and \cite{Charles}), we have
$T(\bar X)_\ell^{\Gal(k')}=0$ for any finite
separable field extension $k'$ of $k$.
\end{proof}

\begin{proof}[Proof of Theorem \ref{tatefiniteK3}]
The first statement is already proved in Theorem \ref{finiteK3} (ii).
As recalled in the introduction, for all but finitely
many primes $\ell$ we have a direct sum decomposition
of $\Gal(k)$-modules
$$\H_{\et}^2(\bar{X},\Z_{\ell}(1))=(\NS(X)\otimes \Z_{\ell}) \oplus 
T(\bar X)_{\ell}.$$
For these primes $\ell$ the Galois module
$\H_{\et}^2(\bar{X},\mu_{\ell})$ is
isomorphic to the direct sum of $\NS(X)/\ell$ and $T(\bar X)_{\ell}/\ell$.
By Theorem \ref{finiteK3} (iii), for all but finitely many $\ell$
we have $(T(\bar X)_{\ell}/\ell)^{\Gal(k)}=0$. This implies
$$\H_{\et}^2(\bar{X},\mu_{\ell})^{\Gal(k)}=(\NS(X)/\ell)^{\Gal(k)}.$$
The cokernel of the natural map $\NS(\bar X)^{\Gal(k)}\to
(\NS(X)/\ell)^{\Gal(k)}$ is a subgroup of the Galois
cohomology group $\H^1(k,\NS(\bar X))$. This group is finite,
since $\NS(\bar X)$ is a free abelian group of finite rank.
Removing the primes that divide the order of $\H^1(k,\NS(\bar X))$
we see that the natural map 
$$\NS(X)=\NS(\bar X)^{\Gal(k)}\lra \H_{\et}^2(\bar{X},\mu_{\ell})^{\Gal(k)}$$
is surjective for all but finitely many primes $\ell$.
\end{proof}

\begin{proof}[Proof of Theorem \ref{BrauerF}]
The Hochschild--Serre spectral sequence
$$\H^p(k,\H^q(\bar X,\G_m))\Rightarrow \H^{p+q}(X,\G_m)$$
gives rise to the well known filtration 
$\Br_0(X)\subset \Br_1(X)\subset\Br(X)$,
where $\Br_0(X)$ is the image of the natural map
$\Br(k)\to\Br(X)$ and $\Br_1(X)$ is the kernel of 
the natural map $\Br(X)\to\Br(\bar X)$. The spectral
sequence shows that $\Br_1(X)/\Br_0(X)$ is contained
in the finite group $\H^1(k,\NS(\bar X))$,
and $\Br(X)/\Br_1(X)$ is contained in $\Br(\bar X)^{\Gal(k)}$.
Hence it is enough to prove that $\Br(\bar X)^{\Gal(k)}$ is finite.

For a prime $\ell\not=p$ the
Kummer exact sequence gives rise to the following
exact sequence of $\Gal(k)$-modules \cite[(5), p. 486]{SZ}:
\begin{eqnarray*}
0\to&(\NS(\bar X)/\ell)^{\Gal(k)}\to 
\H^2(\bar X,\mu_\ell)^{\Gal(k)}& 
\to \Br(\bar X)_\ell^{\Gal(k)}\to\\
&\H^1(k,\NS(\bar X)/\ell) \to 
\H^1(k,\H^2(\bar X,\mu_\ell)).&
\end{eqnarray*}
In the proof of Theorem \ref{tatefiniteK3} we have seen that
for all but finitely many primes $\ell$ the Galois module $\NS(\bar X)/\ell$
is a direct summand of $\H^2(\bar X,\mu_\ell)$, and the second
arrow in the displayed exact sequence is an isomorphism.
It follows that $\Br(\bar X)_\ell^{\Gal(k)}=0$ 
for all but finitely many $\ell$. 

To finish the proof we note that for any prime $\ell\not=p$ the $\ell$-primary
subgroup $\Br(\bar X)\{\ell\}^{\Gal(k)}$ is finite.
This follows from \cite[Prop. 2.5 (c)]{SZ} in view of the
semisimlicity of $\H^2(\bar X,\Q_\ell(1))$, proved in 
Theorem \ref{finiteK3} (i), and the Tate conjecture
$$\NS(X)\otimes \Q_{\ell} = \H_{\et}^2(\bar{X},\Q_{\ell}(1))^{\Gal(k)}$$
proved by Madapusi Pera \cite[Thm. 1]{Pera} for any field $k$ of odd
characteristic $p$ finitely generated over $\F_p$.
\end{proof}

\end{document}